\documentclass{article}

\usepackage{amsmath, amsthm}
\usepackage{amssymb}
\usepackage{color}
\usepackage{graphicx}

\newcommand{\CC}{\mathbb{C}}
\newcommand{\ch}{\mathrm{ch}}
\newcommand{\RR}{\mathbb{R}}
\newcommand{\NN}{\mathbb{N}}
\newcommand{\Schrodinger}{Schr$\ddot{\mathrm{o}}$dinger }
\newcommand{\Hormander}{H$\ddot{\mathrm{o}}$rmander}

\newtheorem{thm}{Theorem}[section]
\newtheorem{prop}[thm]{Proposition}
\newtheorem{lemma}[thm]{Lemma}
\newtheorem{hypo}[thm]{Hypothesis}
\newtheorem{deff}[thm]{Definition}

\title{Quantitative Photo-acoustic Tomography with Partial Data}
\author{Jie Chen and Yang Yang}
\date{\today}

\begin{document}

\maketitle

\begin{abstract}
  Photo-acoustic tomography is a newly developed hybrid imaging modality that combines a high-resolution modality with a high-contrast modality. We analyze the reconstruction of diffusion and absorption parameters in an elliptic equation and improve an earlier result of Bal and Uhlmann \cite{idpat} to the partial date case. We show that the reconstruction can be uniquely determined by the knowledge of 4 internal data based on well-chosen partial boundary conditions. Stability of this reconstruction is ensured if a convexity condition is satisfied. Similar stability result is obtained without this geometric constraint if 4n well-chosen partial boundary conditions are available, where  $n$ is the spatial dimension. The set of well-chosen boundary measurements is characterized by some complex geometric optics (CGO) solutions vanishing on a part of the boundary.
\end{abstract}

\section{Introduction}
Typical medical imaging modalities like computerized tomography(CT), magnetic resonance imaging(MRI) and ultrasound imaging(UI) have high resolution but low contrast. On the other hand, some modalities, based on optical, elastic, or electrical properties of tissues, exhibit a sufficiently high contrast between different types of tissues, but involve a highly smoothing measurement operator and are thus typically low-resolution. Such examples are optical tomography(OP), electrical impedance tomography(EIT) and elastographic imaging(EI).

Hybrid methods combine high resolution with high contrast, based on physical coupling mechanisms. Photo-acoustic tomography(PAT) and thermo-acoustic tomography(TAT) are recent hybrid methods based on the photo-acoustic effect which couples optical and ultrasonic waves. When a body is exposed to short pulse radiation, it absorbs energy and expands thermo-elastically. The expansion emits acoustic pulses, which travel to the boundary of the domain of interest where they are measured. What distinguishes the two modalities is that in PAT, radiation is high-frequency radiation (near-infra-red with sub-$\mu m$ wavelength), while in thermo-acoustic, radiation is low-frequency radiation (microwave with wavelengths comparable to 1$m$).

The first step in both PAT and TAT is the reconstruction of the absorbed radiation, or deposited energy, from time-dependent boundary measurements of acoustic signals. Acoustic signals propagate in fairly homogeneous domains as the sound speed is assumed to be known. The reconstruction of the amount of deposited energy is therefore quite accurate in practical settings. This step has been studied extensively in the mathematical literature, see, e.g. \cite{Ammari2009, Cox2007, Cox2009, CoxLaufer2009, Finch2009, Finch2004, Finch2007, Halt2004, Halt2005, Hristova2008, Kuchment2008, Patch2007, Stefanov2009, Xu2006, Xu2009}.

The second step in PAT or TAT is called quantitative photo- or thermo-acoustics and consists of reconstructing the attenuation and diffusion coefficients from knowledge of the amount of absorbed radiation. This second step is different in PAT and TAT as radiation is typically modeled by transport or diffusion equations in the former case and Maxwell's equations in the latter case. The problem of interest in the paper is quantitative photo-acoustics(QPAT).

In \cite{idpat}, G. Bal and G. Uhlmann studied the uniqueness and stability in the reconstruction of the attenuation and diffusion coefficients in QPAT. In this paper, we extend Bal and Uhlmann's results to partial boundary illumination conditions and prove the uniqueness and stability of the reconstruction. Particularly, we show that two coefficients in the diffusion equation are uniquely determined by four well-chosen illuminations at part of the boundary of the domain and presented an explicit reconstruction procedure. The stability of the reconstruction is established from either four internal data under geometric conditions of strict convexity on the domain of interest, or from $4n$ well-chosen partial boundary conditions, where $n$ is the spatial dimension. 

Mathematically, by the standard Liouville change of variables, the diffusion equation is  replaced by a \Schrodinger equation with unknown potential and with internal and partial boundary measurements. By adapting the complex geometrical optics(CGO) solutions proposed in \cite{CalderonPartial}, we are able to obtain uniqueness and stability results for the inverse \Schrodinger problem. The inverse Liouville change of variables concludes the uniqueness and stability of the the diffusive regime.

The rest of the paper is structured as follows. Section 2 presents the PAT problem and our main results. The inverse \Schrodinger problem and explicit reconstruction algorithms are considered in section 3 and 4. The final results on the inverse diffusion problem are proved in section 5.

\section{Inverse diffusion problem and main results}

	In photo-acoustic diffusion regime, photon propagation is modeled
	by the second-order elliptic equation,
	\begin{equation}
		\label{diff_eq}
		\begin{array}{rrll}
			-\nabla\cdot D(x)\nabla u + \sigma_a(x)u &=& 0 
			&\quad \mathrm{ in }\;\, \Omega\\[3pt]
			u &=& g &\quad \mathrm{ on }\;\, \partial\Omega,
		\end{array}
	\end{equation}
	where $\Omega$ is an open, bounded, connected domain in $\RR^n$
	with $C^2$ boundary $\partial\Omega$, $D(x)$ is a diffusion 
	coefficient, $\sigma_a(x)$ is an attenuation coefficient, and
	$g$ is the prescribed illumination source on $\partial\Omega$.
	In this paper, we will consider the partial data problem, i.e.,	
	$g\in C^{k,\alpha}(\Omega)$ is supported only on 
	a subset of the boundary. Precisely, 
  let $x_0\in\RR^n\backslash \overline{\ch(\Omega)}$, where
  $\ch(\Omega)$ denotes the convex  hull of $\Omega$. We define 
  the front and back sides of   $\partial\Omega$ with respect to 
  $x_0$ by
  \begin{equation}
     \label{frontbackbd}
     \partial\Omega_{\pm}  = 
		      \{x\in\partial\Omega:\pm(x_0-x)\cdot n(x)> 0\},
  \end{equation}
  where $n(x)$ is the unit exterior norm at $x$. 
  Let $\Gamma$ be an open subset of $\partial\Omega$, such that 
  $\partial\Omega_+\subset\Gamma$.  We assume that 
  $\mathrm{supp}(g)\subseteq\Gamma$.

	The energy deposited by the radiation results thermal expansion,
	which then emits acoustic pressure wave $p(t,x)$. The energy 
	deposited is given by
	\begin{equation}
		d(x) = G(x)\sigma_a(x)u
		\qquad\mathrm{ in }\;\, \Omega,
	\end{equation}
	where $G(x)$ is the Gr$\ddot{\mathrm{u}}$neisen
	coefficient. The acoustic	wave is modeled by
	\begin{equation}\left\{
		\begin{array}{rll}
			\partial_{tt}p -\Delta p &=& 0 
			 \qquad \mathrm{ in }\;\, (0,T)\times\RR^n,\\[3pt]
			\partial_t p|_{t=0} &=& 0,  \\[3pt]
			p|_{t=0} &=& f, 
		\end{array}\right.
	\end{equation}
	where $c(x)$ is the speed of the acoustic wave, 
	$T>0$ is fixed, and $f$ is the unknow initial pressure,
	which is proportional to the deposited energy.
	The pressure $p(t,x)$ is then measured on $[0,T]\times
	\partial\Omega$, which allows us to reconstruct $f$ 
	by solving a well-posed inverse wave problem, and thus to 
	construct the 
	energy deposited;, see e.g. \cite{Finch2004, Hristova2008, 
	Kuchment2008, Stefanov2009, Xu2009}.
	We assume this
	step and also assume $G(x)$ in known. The inverse problem
	of interest in this paper is to reconstruct the 
	coefficients $(D(x), 
	\sigma_a(x))$ from the partial boundary illuminations
	$g_j$ and the internal measurements $d_j$, 
	for	$1\leq j \leq J$, $J\in\NN^*$, with $u_j$ solving
	\begin{equation}
		\begin{array}{rrll}
			-\nabla\cdot D(x)\nabla u_j + \sigma_a(x)u_j &=& 0 
			&\quad \mathrm{ in }\;\, \Omega\\[3pt]
			u_j &=& g_j &\quad \mathrm{ on }\;\, \partial\Omega.
		\end{array}
	\end{equation}
	
	The main purpose of this paper is to prove the uniqueness
	and the stability of the coefficient reconstruction. We
	define the set of coefficients $(D(x),\sigma_a(x))\in
	\mathcal{M}$ as
	\begin{equation}
		\mathcal{M} = \{(D(x),\sigma_a(x)) : 
		(\sqrt{D},\sigma_a)\in Y\times C^{k+1}(\bar\Omega),
		\|\sqrt{D}\|_Y + \|\sigma_a\|_{C^{k+1}(\bar\Omega)}
		\leq M\},
	\end{equation}
	Where $Y=H^{n/2+k+2+\varepsilon}(\bar\Omega)\subseteq 
	C^{k+2}(\bar\Omega)$ and $M>0$ is fixed.

	We define a subset
	$\tilde\Omega\subseteq\Omega$ to be the complement of a neighborhood
	of $\partial\Omega_-$ in $\Omega$, i.e.,
	$\tilde\Omega = \Omega\backslash(\mathcal{N}
	(\partial\Omega_-))$, where $\mathcal{N}
	(\partial\Omega_-)$ is a neighborhood of
	$\partial\Omega_-$ in $\Omega$.
	
	The main results for the inverse diffusion problem with internal
	data and partial boundary data are as follows, where
	the measurements $g$ and $d$ are real-valued.
	
	\begin{thm}\label{main1}
		Let $\Omega$ be an open, bounded, connected 
		domain with $C^2$ boundary 
		$\partial\Omega$. Let $\Gamma$ and $\tilde\Omega$ be 
		defined
		as above. Assume that $(D(x),\sigma_a(x))$ and $(\tilde{D}(x),
		\tilde\sigma_a(x))$ are in $\mathcal{M}$ with
		$D|_{\Gamma} = 	\tilde{D}|_{\Gamma}$. 
		Let	$d=(d_j)$ and $\tilde{d}=(\tilde d_j)$,
		$j=1,\ldots,4$, be the internal data
		for coefficients $(D(x),\sigma_a(x))$ and $(\tilde{D}(x),
		\tilde\sigma_a(x))$, respectively and with boundary conditions 
		$g=(g_j)$, $j=1,\ldots,4$. Then there is a set of
		illuminations $g\in(C^{k,\alpha}(\partial\Omega))^4$,
		$\mathrm{supp}(g)\subseteq\Gamma$, for 
		some $\alpha > 1/2$, such that if $d = \tilde{d}$, then
		$(D(x),\sigma_a(x)) = (\tilde{D}(x), \tilde\sigma_a(x))$
		in $\tilde\Omega$.
	\end{thm}
	
	To consider the stability of the reconstruction, additional
	geometric information about $\partial\Omega$ is needed. We impose
	the following hypothesis.
	\begin{hypo}
		\label{ConvHypo}
		Let $\Omega$ and $\partial\Omega_+$ be as above.
		There exists $R<\infty$ such that for each 
		$y\in\partial\Omega_+\subseteq
		\partial\Omega$, we have $\Omega\subset B_{y}(R)$, where
		$B_{y}(R)$ is a ball of radius $R$ that is tangent to $\partial\Omega_+$
		at $y$.
	\end{hypo}
	
	The stability result follows.
	
	\begin{thm}\label{main2}
		Let $k\geq3$. Let $\Omega$ satisfy Hypothesis \ref{ConvHypo}
		with $\partial\Omega$ of class $C^{k+1}$. Let 
		$\Gamma$ and $\tilde\Omega$ be defined as above. 
		Assume that 
		$(D(x),\sigma_a(x))$ and $(\tilde{D}(x), \tilde\sigma_a(x))$
		are in $\mathcal{M}$ with $D|_\Gamma = 
		\tilde{D}|_\Gamma$. Let $d=(d_j)$ and
		$\tilde{d}=(\tilde d_j)$,	$j=1,\ldots,4$, be the internal data
		for coefficients $(D(x),\sigma_a(x))$ and $(\tilde{D}(x),
		\tilde\sigma_a(x))$, respectively and with boundary conditions 
		$g=(g_j)$, $j=1,\ldots,4$. Then there is a set of
		illuminations $g\in(C^{k,\alpha}(\partial\Omega))^4$,
		$\mathrm{supp}(g)\subseteq\Gamma$, for some $\alpha
		> 1/2$ and a constant $C$ such that
		\begin{equation}
			\|D-\tilde D\|_{C^{k-1}(\tilde{\Omega})} +
			\|\sigma_a-\tilde \sigma_a\|_{C^{k-1}(\tilde{\Omega})} \leq
			C\|d-\tilde d\|_{(C^{k}(\tilde{\Omega}))^4}.
		\end{equation}
	\end{thm}
	
	Instead of imposing geometric hypothesis, the stability of the 
	reconstruction could also follow from $4n$ well-chosen measurements.
	
	\begin{thm}\label{main3}
		Let $k\geq2$. Let $\Omega$ be an arbitrary bounded domain
		with $\partial\Omega$ of class $C^{k+1}$. Let $\Gamma$
		and	$\tilde\Omega$ be defined as above. Assume that 
		$(D(x),\sigma_a(x))$ and $(\tilde{D}(x), \tilde\sigma_a(x))$
		are in $\mathcal{M}$ with $D|_\Gamma = 
		\tilde{D}|_\Gamma$. Let $d=(d_j)$ and 
		$\tilde{d}=(\tilde d_j)$,
		$j=1,\ldots,4n$, be the internal data
		for coefficients $(D(x),\sigma_a(x))$ and $(\tilde{D}(x),
		\tilde\sigma_a(x))$, respectively and with boundary condition 
		$g=(g_j)$, $j=1,\ldots,4n$. Then there is a set of
		illuminations $g\in(C^{k,\alpha}(\partial\Omega))^{en}$,
		$\mathrm{supp}(g)\subseteq\Gamma$, for 
		some $\alpha > 1/2$ and a constant $C$ such that
		\begin{equation}
			\|D-\tilde D\|_{C^{k}(\tilde{\Omega})} +
			\|\sigma_a-\tilde \sigma_a\|_{C^{k}(\tilde{\Omega})} \leq
			C\|d-\tilde d\|_{(C^{k+1}(\tilde{\Omega}))^{4n}}.
		\end{equation}
	\end{thm}
	
	The proof of these results are mainly based on the corresponding
	results for the inverse \Schrodinger equation. In the following
	section, we will focus on \Schrodinger equation and prove similar
	results.

\section{Inverse \Schrodinger equation with internal data}
\label{ISID}
	Let $\Omega$ be an open, bounded, connected domain in $\RR^n$ with
  smooth boundary $\partial\Omega$. Let $\partial\Omega_+$ be 
  defined in (\ref{frontbackbd}) and let
  $\Gamma\subset\partial\Omega$ be an open subset such that
  $\partial\Omega_+\subset\Gamma$.
  Consider the diffusion equation with unknow diffusion coefficient
  $D$ and unknow attenuation coefficient $\sigma_a$:
  \begin{equation*}
      -\nabla\cdot D\nabla u + \sigma_a u = 0, 
      \quad\mathrm{in}\;\,\Omega,
      \qquad
      u = g, \quad \mathrm{on}\;\,\partial\Omega,
  \end{equation*}
  where $\mathrm{supp}(g)\subseteq \Gamma$.  
  Using the standard Liouville change of variables, $v = \sqrt{D}u$
  solves
  \begin{equation}
  		\label{Liou1}
      \Delta v+qv=0,
  \end{equation}
  with
  \begin{equation}
  		\label{Liou2}
      q = -\frac{\Delta\sqrt{D}}{\sqrt{D}} - \frac{\sigma_a}{D}.
  \end{equation}
  The internal data in photo-acoustics is  given by
  \begin{equation}
  		\label{Liou3}
      d = \sigma_a u = \frac{\sigma_a}{\sqrt{D}}v = \mu v,
      \quad
      \mu:=\frac{\sigma_a}{\sqrt{D}},
  \end{equation}
  while the new boundary condition is given by $\sqrt{D}g$ on
  $\partial\Omega$, assuming $D$ is known on 
  $\partial\Omega$.
  
  After relabeling, we have a \Schrodinger equation of the form
  \begin{equation}
  	\label{SchrEq}
		\begin{array}{rrll}
			\Delta u_j + qu_j &=& 0 
			&\quad \mathrm{ in }\;\, \Omega\\[3pt]
			u_j &=& g_j &\quad \mathrm{ on }\;\, \partial\Omega,
		\end{array}
	\end{equation}
	where $1\leq j\leq J$, $J\in\NN^*$ is the number of illuminations, 
	$q$ is an unknown potential, and 
	$\mathrm{supp}(g_j)\subseteq \Gamma$.
	We assume that the homogeneous
	problem with $g_j=0$ admits the unique solution $u_j\equiv 0$
	so that $\lambda=0$ is not in the spectrum of $\Delta+q$. We
	assume that $q$ on $\Omega$ is the restriction to $\Omega$ of 
	a function $\tilde{q}$ compactly supported on $\RR^n$ and such
	that $\tilde{q}\in H^{n/2+k+\varepsilon}(\RR^n)$ with 
	$\varepsilon>0$ for $k\geq 1$.
	Moreover, we assume that the extension is chosen so that
	\begin{equation}
		\|q\|_{H^{n/2+k+\varepsilon}(\Omega)}\leq C
		\|\tilde{q}\|_{H^{n/2+k+\varepsilon}(\RR^n},
	\end{equation}
	for some constant $C=C(\Omega, k, n)$ independent of $q$. That
	such a constant exists may be found e.g. in 
	\cite{Stein1970}, Chapter VI, Theorem 5.
	
	We assume that $g_j\in C^{k,\alpha}(\partial\Omega)$ with $\alpha>
	\frac{1}{2}$ and $\partial\Omega$ is of class $C^{k+1}$ so that
	(\ref{SchrEq}) admits a unique solution $u_j\in C^{k+1}(\Omega)$,
	see \cite{Gilbarg1977}, Theorem 6.19. The internal data are of the 
	form
	\begin{equation}
		d_j(x) = \mu(x)u_j(x),\qquad \mathrm{ in }\;\, \Omega,
		1\leq j \leq J.
	\end{equation}
	Here $\mu\in C^{k+1}(\bar{\Omega})$ verifies $0<\mu_0\leq
	\mu(x)\leq\mu_0^{-1}$.
	
	The \emph{partial data
	inverse Schr$\ddot{o}$dinger problem with internal data}
	(PISID) consists of reconstructing $(q,\mu)$ in $\Omega$ from
	knowledge of $d=(d_j)_{1\leq j\leq J}\in(C^{k+1}(\Omega))^J$ and
	illuminations $g=(g_j)_{1\leq j\leq J}\in (C^{k,\alpha}(\partial
	\Omega))^J$, while $\mathrm{supp}(g_j)\subseteq \Gamma$.
	We will mostly be concerned with the case $J=4$ and $J=4n$ with
	$g_j$ and	$d_j$ real-valued measurements.

\subsection{Construction of complex geometrical optics solution}    

	The proof mainly depends on the \emph{Complex Geometrical Optics}(CGO) solutions. Let $P_0 = -h^2\Delta = \sum (hD_{x_k})^2$ be a differential operator, with $h$ a small parameter and $D_{x_k}=-i\partial_{x_k}$. Let $\varphi\in C^\infty(\Omega;\RR)$, with $\nabla\varphi\neq 0$ everywhere. Consider the conjugated operator
	\begin{equation}
		e^{\varphi/h}\circ P_0\circ e^{-\varphi/h} 
		= \sum_{k=1}^n (hD_{x_k} + i\partial_{x_k} \varphi)^2
		= A + i B,
	\end{equation}
	where $A, B$ are the formally selfadjoint operators,
	\begin{eqnarray*}
		A &=& (hD)^2 - (\nabla\varphi)^2,\\
		B &=& \nabla\varphi\cdot hD + hD\cdot\nabla\varphi,
	\end{eqnarray*}
	having the principal symbols
	\begin{equation}
		\label{PricSym}
		a =  \xi^2 - (\nabla\varphi)^2,
		\qquad
		b =  2\nabla\varphi\cdot\xi.
	\end{equation}
	
	We want the conjugated operator $e^{\varphi/h}\circ P_0\circ e^{-\varphi/h}$ to be locally solvable in a semi-classical sense, which means its principal symbol satisfies \Hormander's condition
	\begin{equation}
		\label{Hormander}
		\{a,b\} = 0,\qquad\mathrm{when }\;\, a = b =0,
	\end{equation}
	where $\{a, b\}=\sum_k(\partial_{\xi_k} a\partial_{x_k}b - \partial_{x_k} a \partial_{\xi_k} b)$ is the Poisson bracket.
	
	\begin{deff}
		A real smooth function $\varphi$ on an open set $\Omega$ is said to be a limiting Carleman weight if it has non-vanishing gradient on $\Omega$ and if the symbols (\ref{PricSym}) satisfy the condition $(\ref{Hormander})$. This is equivalent to say that 
		\begin{equation}
			\langle \varphi''\nabla\varphi,\nabla\varphi\rangle 
			+ \langle\varphi''\xi,\xi\rangle =0
			\quad when\;\;
			\xi^2 = (\nabla\varphi)^2
			\;\;and\;\;
			\nabla\varphi\cdot\xi=0,
		\end{equation}
		where $\phi''$ is the Hessian matrix of $\phi$.
	\end{deff}

%

  Let $H^s(\Omega)$ be the Sobolev space. We denote by $H^1_{scl}(\Omega)$ the semi-classical Sobolev space of order $1$ on $\Omega$ with associated norm
  \begin{equation}
  	\|u\|^2_{H^1_{scl}(\Omega)}=\|h\nabla u\|^2 + \|u\|^2
  \end{equation}
and by $H^s_{scl}(\RR^n)$ the semi-classical Sobolev space of order $s$ on $\RR^n$ with associated norm
	\begin{equation}
		\|u\|^2_{H^s_{scl}(\RR^n)}=\|\langle hD
		  \rangle^su\|^2_{L^2(\RR^n)}
		=\int (1+h^2\xi^2)^s|\hat{u}(\xi)|^2 d\xi.
	\end{equation}
Let $P=P_0 + h^2q$. Carleman estimates give the following solvablity result \cite{Ferreira2006, CalderonPartial}:
\begin{prop}
  \label{Hahn-Banach}
  	Let $\varphi$ be a limiting Carleman weight and $q\in L^\infty(\Omega)$. Let $0\leq s\leq 1$. There exists $h_0>0$ such that for $0< h\leq h_0$ and for every $w\in H^{s-1}_{scl}(\Omega)$, there exists $u\in H^s_{scl}(\Omega)$ such that
  	\begin{equation*}
  		e^{\frac{\varphi}{h}} P e^{-\frac{\varphi}{h}}u = v,
  		\quad
  		h\|u\|_{H^s_{scl}(\Omega)}\leq C\|v\|_{H^{s-1}_{scl}(\Omega)}.
  	\end{equation*}
  \end{prop}

  Now we start to construct CGO solutions by choosing $\varphi\in C^\infty(\Omega)$ to be a limiting Carleman weight and $\psi\in C^\infty(\Omega)$ such that
  \begin{equation}
  	(\nabla\psi)^2 = (\nabla\varphi)^2, \quad
  	\nabla\psi\cdot\nabla\varphi = 0.
  \end{equation}
  Then $\psi$ is a local solution to the Hamilton-Jacobi problem
  \begin{equation}
  	a(x, \nabla\psi) = b(x, \nabla\psi) = 0.
  \end{equation}
  Therefore,
  \begin{eqnarray}
  	\label{P0conj}
  	e^{-\frac{1}{h}(-\varphi+i\psi)} P_0
  	  e^{\frac{1}{h}(-\varphi+i\psi)}a
  	&=& \big[((hD +\nabla\psi)^2 -\nabla\varphi^2) 
  	  + i\big(\nabla\varphi(hD+\nabla\psi)\notag\\
  	&&  +(hD+\nabla\psi)\nabla\varphi\big)\big]a\notag\\
  	&=& (hL - h^2\Delta)a,
  \end{eqnarray}
  where $L$ is the transport operator given by
  \begin{equation}
  	L = \nabla\psi D + D\nabla\psi 
  	    + i(\nabla\varphi D + D\nabla\varphi).
  \end{equation}
  
  There exists a non-vanishing smooth function 
  $a\in C^\infty(\Omega)$, see \cite{Duist1972, CalderonPartial},
  such that
  \begin{equation}
  	La = 0.
  	\label{a:constrain}
  \end{equation}
  
Assume that $q\in L^\infty(\Omega)$ and recall $P = h^2(\Delta + q) = P_0 + h^2q$. Then (\ref{P0conj}) implies that
  \begin{equation*}
  	Pe^{\frac{1}{h}(-\varphi + i\psi)}a = e^{-\varphi/h}h^2\kappa,
  \end{equation*}
with $\kappa = \mathcal{O}(1)$ in $L^\infty$ and hence in $L^2$. Now Proposition \ref{Hahn-Banach} implies there exists $r(x,h)\in H^1_{scl}(\Omega)$ such that
  \begin{equation}
  	\label{r:constrain}
  	\begin{array}{c}
	  	e^{\varphi/h}P e^{\frac{1}{h}(-\varphi + i\psi)}r = -h^2\kappa,
	  	\quad\mathrm{and}\\
	  	\|r\|_{H^1_{scl}(\Omega)}=\|h\nabla r\| +\|r\| \leq Ch, 
	  	\quad\mathrm{for\;some\;} C.
  	\end{array}
  \end{equation}
Hence,
	\begin{equation*}
		P(e^{\frac{1}{h}(-\varphi + i\psi)}(a+r)) = 0,
	\end{equation*}
i.e., we constructed a solution of the \Schrodinger equation of the form
	\begin{equation}
		\label{CGO_sol}
		e^{\frac{1}{h}(-\varphi + i\psi)}(a+r).
	\end{equation}

  In our partial data model, we illuminate on part of the boundary, i.e., $\mathrm{supp}(g)\subseteq \Gamma$. But the CGO solution in (\ref{CGO_sol}) is non-vanishing everywhere. Fortunately, using \cite{CalderonPartial}, we can modify the CGO solution to match our model.
  
  Define $\Gamma_-=\partial\Omega \backslash\Gamma$. Then, $\Gamma_-$ is a strict open subset of $\partial\Omega_-$.
  \begin{prop}
      \label{CGOprop}
      We can construct a solution of
      \begin{equation}
          Pu = 0,\quad u\big|_{\Gamma_-} = 0
      \end{equation}
      of the form
      \begin{equation}
          \label{CGOsolution}
          u = e^{\frac{1}{h}(\varphi+i\psi)}(a + r) + z
      \end{equation}
      where $\varphi, \psi$ and $a$ are chosen as above and 
      $z = e^{i\frac{l}{h}}b(x;h)$ with $b$ a symbol of order
      zero in $h$ and
      \begin{equation*}
          \mathrm{Im}\, l(x) = -\varphi(x) + k(x)
      \end{equation*}
      where $k(x) \sim \mathrm{dist}(x,\partial\Omega_-)$ in a
      neighborhood of $\partial\Omega_-$ and $b$ has its support 
      in that neighborhood. Moreover, $\|r\|_{H^0} = 
      \mathcal{O}(h)$,
      $r|_{\partial\Omega_-} = 0$, 
      $\|(\nabla\varphi\cdot\nu)^{\frac{1}{2}}r\|_{\partial\Omega_+}
      = \mathcal{O}(h^{\frac{1}{2}})$.
  \end{prop}
  
  Remark: According to the proof in \cite{CalderonPartial}, $\mathrm{supp}(z)$ is arbitrarily close to $\partial\Omega_-$ in $\Omega$.

  Notice that $\varphi$ is a limiting Carleman weight, so is $-\varphi$. We construct the CGO solutions  of the form
  \begin{equation}\label{CGOu12}
  \begin{array}{rl}
      \check{u}_1 &= e^{\frac{1}{h}(\varphi+i\psi)}
                    (a_1 + r_1) + z_1\\[3pt]
      \check{u}_2 &= e^{\frac{1}{h}(-\varphi+i\psi)}
                    (a_2 + r_2) + z_2.
  \end{array}\end{equation}
	In particular, we choose
  \begin{equation}
  \label{phipsi_def}
      \varphi(x) = \log|x-x_0|\quad\mathrm{and}\quad
      \psi(x) = d_{S^{n-1}}\left(\frac{x-x_0}{|x-x_0|},
                \omega\right),
  \end{equation}
  where $x_0\in \RR^n\backslash\overline{\ch(\Omega)}$ and 
  $\omega\in S^{n-1}$.

  Note that $z_j$, $j=1,2$, are supported only in an arbitrarily  neighborhood of $\partial\Omega_-$. For the rest of this paper, we will mainly   consider the uniqueness and the stability of the solutions on the  subregion defined by 
  \begin{equation}
      \label{tildeomega}
      \tilde{\Omega} = \Omega \backslash 
      (\mathrm{supp}(z_1)\cup \mathrm{supp}(z_2)).
  \end{equation}
  Remark that a neighborhood of the point $y\in\partial\Omega_+ \cap \partial\Omega_- = \{y\in\partial\Omega|-(y-x_0)\cdot n(y) = 0\}$ is excluded from $\tilde\Omega$. Thus, there exists a fixed constant $\eta>0$, such that  if $x\in\tilde\Omega$ and $y\in\partial\Omega_+\cap\partial\tilde\Omega$, then $-(x-x_0)\cdot n(y) \geq\eta$. In the following proof, we mainly focus on $\tilde\Omega$, where $z_1$ and $z_2$ vanish. 

\subsection{Construction of vector fields and uniqueness result}
  
  We assume that we can impose the complex-valued illuminations 
  $g_j$  
  on $\partial\Omega$ with
  $\mathrm{supp}(g_j)\subseteq\Gamma$ and observe the 
  complex-valued internal data $d_j$, $j=1,2$. Note that, to make up 
  complex-valued $g_j$ and $d_j$, $j=1,2$, we need to four real 
  observations. $d_{j}$ are of the form $d_j = \mu u_j$
  in $\Omega$, where $u_{j}$ is the solutions of
  \begin{equation}
      \label{scheqn2}
  		\Delta u_j + qu_j = 0, \;\;\mathrm{in }\; \Omega \qquad
  		u_j = g_j, \;\; \mathrm{on }\; \partial\Omega, \qquad
  		j = 1, 2.
  \end{equation} 
  Direct calculation gives us that
  $$
      u_1\Delta u_2 - u_2 \Delta u_1 = 0.
  $$

  We assume that $\mu\in C^{k+1}(\bar{\Omega})$ is bounded above 
  and below by positive constants. By substituting $u_j = 
  {d_j}/{\mu}$, we obtain that
  $$
      2(d_1\nabla d_2 - d_2\nabla d_1)\cdot\nabla\mu
      - (d_1\Delta d_2 - d_2\Delta d_1)\mu = 0,
  $$
  or equivalently,
  \begin{equation}
      \label{vectorode}
      \beta_d\cdot\nabla \mu + \gamma_d \mu = 0,
  \end{equation}
  where
  \begin{equation}
      \label{betagamma}
      \begin{array}{rcl}
      \displaystyle \beta_d &:=& \displaystyle
      \chi(x)(d_1\nabla d_2 - d_2\nabla d_1) 
      \\[6pt]
      \displaystyle \gamma_d &:=& \displaystyle
      -\frac{1}{2}\chi(x)(d_1\Delta d_2 - d_2\Delta d_1)
      =\frac{-\beta\cdot\nabla\mu}{\mu}.
  \end{array}\end{equation}
  Here, $\chi(x)$ is any smooth known complex-valued function
  with $|\chi(x)|$ uniformly bounded below by a positive constant
  on $\bar\Omega$. Note that by assumption on $\mu$, we have that
  $\beta_d\in (C^k(\bar{\Omega};\CC))^n$ and
  $\gamma\in C^k(\bar{\Omega}; \CC)$.
  
  The methodology for the reconstruction of $(\mu, q)$ will
  be the same as in \cite{idpat}: we first reconstruct $\mu$
  using the real part or the imaginary part of (\ref{vectorode})
  and the boundary condition $\mu = d / g$ on 
  $\partial\Omega_+$. 
  Then we may recover $u_{j} = d_{j}/\mu$, $j=1,2$,
  explicitly and therefore
  $q$ from the Schr$\mathrm{\ddot{o}}$dinger equation 
  (\ref{scheqn2}). 
  
  The reconstruction of $\mu$ is unique as long as the integral 
  curves of (the real part or the imaginary part of) $\beta_d$ 
  join any interior point  $x\in\Omega$ to a point on
  $\partial\Omega_+$, denoted as $x_+(x)$. This property of 
  $\beta$ is not guaranteed in general, unless the boundary 
  conditions $g_j$, $j=1,2$, are properly chosen. CGO solutions are 
  then employed to construct a family of boundary conditions
  $g_j$, $j=1,2$, which could produce well-performed vector fields 
  $\beta_d$.
  
  To make a distinction from the observed data, we denote the CGO solutions by
  $\check{u_1}$ and $\check{u_2}$, and denote
  the internal data and the vector field constructed from CGO 
	solutions by $\check{d}$ and $\check{\beta}_d$, respectively. 
  
  \begin{prop} \label{uniquemu}
      Let $\beta$ be the the normalized vector field, defined by
      \begin{equation}
          \label{normbeta}
          \beta = \frac{h}{2}\,\beta_d.
      \end{equation}
      There is a open set of $g$ in $C^{k,\alpha}(\partial\Omega)$, 
      with $\mathrm{supp}(g)\subseteq\Gamma$,
      such that, for any small $\epsilon$, 
      \begin{equation}			  
		      \label{betaapprox}
		      \left\|\beta(x) - \mu^2\,\Gamma
		      \frac{x_0 - x}{|x_0 - x|^2}\right\|_{C^{k}\,
		      (\tilde{\Omega},\CC)}
		      \leq C\,h(1+\epsilon) \quad \mathrm{on}\;\tilde{\Omega},
		  \end{equation}
		  where $x_0\in\RR^n\backslash\overline{\ch(\Omega)}$ and
		  $\Gamma$ is a function of class $C^{k}(\Omega)$.
      Therefore, (\ref{vectorode}) admits a unique solution $\mu$ on
      $\tilde\Omega$.
  \end{prop}

  \begin{proof}
    Let $\check u_1, \check u_2$ be CGO solutions in 
    (\ref{CGOu12}). By 
    substituting $\check{d}_j = \mu\check{u}_j$ and
    $\chi(x) = e^{-\frac{2i}{h}\psi}$ in   
    (\ref{betagamma}), we find $\check{\beta}_d$, 
    restricted to $\tilde{\Omega}$, is given by
	  \begin{align*} 
	      \check{\beta}_d = \mu^2 \Big(
	      &- \frac{2\nabla\varphi}{h}(a_1 + r_1)(a_2 + r_2)
	       + (a_1 + r_1)(\nabla a_2 + \nabla r_2)\\
	      &- (a_2 + r_2)(\nabla a_1 + \nabla r_1)
	      \Big).
	  \end{align*}
	  We may then define, on $\tilde{\Omega}$,
	  \begin{equation}
	      \label{vectorest}
	      \check{\beta} = \frac{h}{2} \check{\beta}_d
	      = - \mu^2\,\Gamma\,\nabla\varphi + \mu^2H,
	  \end{equation}
	  where $\Gamma = (a_1 + r_1)(a_2 + r_2)$ and 
	  $H = \frac{h}{2}(  (a_1 + r_1)(\nabla a_2 + \nabla r_2)
	       - (a_2 + r_2)(\nabla a_1 + \nabla r_1)) \leq C_0 h$ 
	  for some constant $C_0$. Also note that, from (\ref{phipsi_def}), 
	  $\nabla\phi = \frac{x_0-x}{|x_0-x|^2}$.

	  		  
	  We will next choose appropriate boundary conditions 
		$g_j$, $j=1,2$, on
	  $\partial\Omega$, which could lead to small
	  $\|\beta_d - \check{\beta}_d\|$ in $\Omega$.		  
	  In particular, recall that $\check{u}_j=0$ on $\partial
	  \Omega_-$, for some $\epsilon > 0$, we choose $g_j\in 
	  C^{k,\alpha}(\partial\Omega)$ and $g_j=0$ on 
	  $\partial\Omega_-$, such that,
	  \begin{equation}
	    \label{gconstrain} 
	    \|g_j - \check{u}_j\|_{C^{k,\alpha}(\partial\Omega)}
	    \leq \epsilon \quad \mathrm{on}\;\partial\Omega,
	    \quad j = 1,2.
	  \end{equation}
	  Let $u_j$ be the solutions of (\ref{scheqn2}) with 
	  boundary conditions
	  $g_j$ from (\ref{gconstrain}). By elliptic regularity, 
	  we have that
	  \begin{equation}
	  		\label{uconstrain}
	      \|u_j - \check{u}_j\|_{C^{k+1}(\bar{\Omega},\CC)}
	    \leq C \epsilon \quad \mathrm{on}\;\bar{\Omega},
	    \quad j = 1,2,
	  \end{equation}
	  for some positive constant $C$. Notice that $d_j = \mu u_j$
	  and $\mu\in C^{k+1}(\bar{\Omega})$, we deduce that
	  \begin{equation}
	     \label{dapprox}
	     \|d_j - \check{d}_j\|_{C^{k+1}(\bar{\Omega},\CC)}
	     \leq C_0 \epsilon \quad \mathrm{on}\;\bar{\Omega},
	     \quad j = 1,2.
	  \end{equation}
	  
	  Thus, restricting to $\tilde{\Omega}$, (\ref{vectorest})
	  and (\ref{dapprox}) induce (\ref{betaapprox}),
	  which indicates, when $h, \epsilon$ are sufficiently small, 
	  $\beta$ is close to a non-vanishing vector 
	  $-\mu^2\,\Gamma\,\nabla\varphi = 
	  \mu^2\,\Gamma\,\frac{x_0 - x}{|x_0 - x|^2}$. 
	  Approximately, the integral curves of $\beta$ are rays
	  from any $x\in\tilde{\Omega}$ to $x_0\in\RR^n\backslash
	  \overline{\ch(\Omega)}$, intersecting $\partial\Omega_+$
	  at a point $x_+(x)$. Therefore, 
	  with $\mu = d_1/g_2 = d_2/g_2$ known
	  on $\partial\Omega_+$, 
	  \begin{equation}
	      \label{normvecode}
	      \beta\cdot\nabla\mu +\gamma\mu = 0,
	      \quad \gamma = \frac{h}{2}\gamma_d
	  \end{equation}
	  provides a unique reconstruction for $\mu$, so does 
	  (\ref{vectorode}). More precisely, consider the flow
	  $\theta_x(t)$ associated to $\beta$, i.e., the solution to
	  \begin{equation*}
	      \dot{\theta}_x(t) = \beta(\theta_x(t)),
	      \qquad \theta_x(0) = x \in\tilde{\Omega}.
	  \end{equation*}
	  By the Picard-Lindel$\mathrm{\ddot{o}}$f theorem, the above
	  equations admit unique solution $\theta_x$ while $\beta$ is of class
	  $C^1$. Since $\beta$ is non-vanishing, $\theta_x$ reaches
	  $x_+(x)\in\partial\Omega_+$ in a finite time, denoted as 
	  $t_+(x)$, i.e.,
	  \begin{equation*}
	      \theta_x(t_+(x)) = x_+(x).
	  \end{equation*}
	  Then by the method of characteristics, $\mu(x)$ is uniquely
	  determined by
	  \begin{equation}
	      \label{muchar}
	      \mu(x) = \mu_0(x_+(x))e^{-\int_0^{t_+(x)}
	      \gamma(\theta_x(s))ds},
	  \end{equation}
	  where $\mu_0 = d/g$ on $\partial\Omega_+$. This finishes the
	  proof.
  \end{proof}

  Let us define the set of parameters
  \begin{equation}
  		\begin{array}{c}
  		     \mathcal{P} \,=\,\Big\{
  		     (\mu,q)\in C^{k+1}({\Omega})\times
  		     H^{\frac{n}{2}+k+\epsilon}({\Omega});\;
  		     0 \mathrm{\; is\; not\; an\; eigenvalue\; of\; }\,
  		     \Delta + q,\\
  		     \|\mu\|_{C^{k+1}({\Omega})}
  		     +\|q\|_{H^{\frac{n}{2}+k+\epsilon}({\Omega})}
  		     \leq P < \infty
  		     \Big\}.
  		\end{array}
  \end{equation}
  The above construction of the vector field allows us to obtain
  the following uniqueness result.
  
  \begin{thm} \label{uniquenessthm}
      Let $\Omega$ be a bounded, open subset of $\RR^n$ with
      boundary of class $C^{k+1}$. Let $(\mu,q)$ and $(\tilde{\mu},
      \tilde{q})$ be two elements in $\mathcal{P}$ and
      $\tilde{\Omega}$ be defined in $(\ref{tildeomega})$. When
      $h$ and $\epsilon$ are sufficiently small, for $j=1,2$, let
      $g_j$ be constructed according to (\ref{gconstrain})
      with the CGO solutions $\check{u}_j$. 
      $d_j$ and $\tilde{d}_j$ are two sets of
      observations of the internal data on $\Omega$.
      
      Restricting to $\tilde{\Omega}$, $d_j=\tilde{d}_j$ 
      implies that $(\mu,q)=(\tilde{\mu},\tilde{q})$.
  \end{thm}
  
  \begin{proof}
      We have proved that, for $j=1,2$, 
      when $g_j$ is properly chosen, 
      $\mu$ is uniquely reconstructed on $\tilde\Omega$, 
      i.e., $\mu = \tilde{\mu}$. 
      Directly, $d_j = \tilde{d}_j$ also implies
      $u:=u_j = \tilde{u}_j$. By unique continuation,
      $u$ cannot vanish on an open set in $\tilde\Omega$ different 
      from the empty set. Otherwise $u$ vanishes everywhere and 
      this is impossible to satisfy the boundary conditions. 
      Therefore, the set $F=\{|u|>0\}\cap\tilde\Omega$
      is open and $\bar{F} = \bar{\tilde\Omega}$ since the 
      complement of $\bar{F}$ has to be empty. By continuity,
      this shows that $q=\tilde{q}$ on $\tilde\Omega$.
  \end{proof}

  Since the coefficient $q$ in the Schr$\mathrm{\ddot{o}}$dinger 
  equation is unknown, the CGO solutions $\check u_j$, $j=1,2$, 
  cannot be explicitly determined.
  Therefore, although we know that $g_j$ can be chosen from
  a open set close to $\check u_j$, a more explicit
  characterization of the open set is lacking.
  
  Also notice that the parameters $h$ and $\epsilon$ need to be
  small to make $\beta$ flat enough, while cannot be too small,
  otherwise $g_{1}$ will be so large that the imposed illuminations 
  become physically infeasible and
  $g_2$ will be so small that the imposed illuminations 
  become physically undetectable.

\section{Stability of the reconstruction}
	In this section, we consider the stability of the proposed 
  reconstruction method. We divide  the proof into two cases. When
  two complex-valued data are measured, a strict convexity is
  assumed on the domain of interest. In another case, $2n$
  complexed-valued data are measured and the stability result
  follows without any geometric condition on the domain.

\subsection{Stability Result for $2$ complex observations}
  To avoid the case 
  when the vector field is approximately parallel to the boundary 
  $\partial\Omega_+$, i.e., $\beta(x_+(x))\cdot 
  n(x_+(x))$ is small, for a long time, we impose the 
  convexity condition as established in Hypothesis
  (\ref{ConvHypo}).

  Recall that $\theta_x(t)$ is the flow associated  to $\beta$.
  Assume $\theta_x(t)$ reaches the boundary at $x_+$ and at time
  $t_+$, i.e., $\theta_x(t_+) = x_+\in\partial\Omega$. Similar
  notations are use for $\tilde\beta$.
  We have the equality
  \begin{equation*}
      \theta_x(t) - \tilde{\theta}_x(t)
      = \int_0^t [\beta(\theta_x(s)) - 
      \tilde{\beta}(\tilde{\theta}_x(s))]ds.
  \end{equation*}    
  Using the Lipschitz continuity of $\beta$ and Gronwall's
  lemma, we thus deduce the existence of a constant $C$ such
  that   
  $$
      |\theta_x(t) - \tilde{\theta}_x(t)| 
      \leq C t \|\beta-\tilde{\beta}\|_{C^0(\Omega)}
  $$
  uniformly in $t$ knowing that all characteristics exit
  $\Omega$ in finite time and provided that $\theta_x(t)$
  and $\tilde{\theta}_x(t)$ are in $\bar{\Omega}$.
  
  Such estimates are stable with respect to modifications in the
  initial conditions. Let us define $W(t) = D_x\theta_x(t)$. Then 
  $W$ solves the equation $\dot{W} = D_x\beta(\theta_x)$ with 
  $W(0) = I$ and by using Gronwall's lemma once more, we deduce 
  that
  $$
      |W(t) - \tilde{W}(t)| \leq Ct
      \|D_x\beta - D_x\tilde{\beta}\|_{C^0(\Omega)},
  $$
  for all times provided that $\theta_x(t)$ and 
  $\tilde{\theta}_x(t)$ are in $\bar{\Omega}$. As a 
  consequence, since $\beta$ and $\tilde{\beta}$ are in 
  $C^k(\bar{\Omega})$, then we obtain similarly that:
  \begin{equation}\label{thetastable}
      |D_x^k\theta_x(t) - D_x^k\tilde{\theta}_x(t)|
      \leq Ct\|\beta - \tilde{\beta}\|_{C^k(\Omega)},
  \end{equation}
  and this again for all times $\theta_x(t)$ and 
  $\tilde{\theta}_x(t)$ are in $\bar{\Omega}$.
  To simplify the notation, we define 
  $\delta = \|\beta - \tilde{\beta}\|_{C^k(\Omega)}$.
  
  Recall that by the definition of $\tilde{\Omega}$ in (\ref{tildeomega}). Let $x\in\tilde\Omega$ and $y\in\partial\Omega_+\cap\partial\tilde\Omega$, then $\beta(x)\cdot n(y) > \eta$, for some constant $\eta>0$, where $\beta$ satisfies (\ref{betaapprox}) and $n(y)$ is the unit outer normal.
    
  \begin{lemma} \label{xtlemma}
    Let $k\geq 1$ and assume that $\beta$ and $\tilde{\beta}$
    are $C^k(\bar{\Omega})$ vector fields that are sufficiently 
    flat, i.e., $h$ is sufficiently small. Let us assume that 
    $\partial\Omega_+$ is sufficiently convex so that Hypothesis 
    \ref{ConvHypo} holds for some $R<\infty$. Then, restricting 
    to $\tilde{\Omega}$, we have that
    \begin{equation}
        \|x_+ - \tilde{x}_+\|_{C^{k}{(\tilde{\Omega})}}
          +  \|t_+ - \tilde{t}_+\|_{C^{k}{(\tilde{\Omega})}}
        \leq C \|\beta - \tilde{\beta}
             \|_{C^{k}{(\tilde{\Omega})}},
    \end{equation}
	  where $C$ is a constant depending on $h$ and $R$.
  \end{lemma}
  
  \begin{proof}
  Let $x\in\tilde\Omega$. Without loss of generality, we assume $t_+(x)<\tilde{t}_+(x)$. Draw a circle with radius $R$ and tangent to $\Omega$ at $x_+$. We impose a coordinate system by choosing the circle center to be the origin.  Equation (\ref{betaapprox}) and (\ref{thetastable}) gives that
	\begin{equation*}
		|\tilde x_+ - [\tilde\theta_x(t_+) 
		+ \tilde\beta(\tilde\theta_x(t_+))(\tilde t_+ - t_+)]| 
		\leq C (\tilde t_+ - t_+) \delta.
	\end{equation*}
	Thus,
	\begin{equation*}
		|\tilde\theta_x(t_+) 
		+ \tilde\beta(\tilde\theta_x(t_+))(\tilde t_+ - t_+)| 
		\leq R + C (\tilde t_+ - t_+)\delta.
	\end{equation*}
	Directly, we have
	\begin{align}
		&|x_+ + \tilde\beta(\tilde\theta_x(t_+))(\tilde t_+ - t_+)|^2
		 - R^2\notag\\
		&\qquad= |\tilde\theta_x(t_+) 
		+ \tilde\beta(\tilde\theta_x(t_+))(\tilde t_+ - t_+) 
		- (\tilde\theta_x(t_+) - x_+)|^2 - R^2 \notag\\
		&\qquad= K [( R + C (\tilde t_+ - t_+)\delta 
		+ |\tilde\theta_x(t_+) - x_+|)^2 - R^2] . \label{lemmaeq1}
	\end{align}
	where $0<K\leq 1$ is defined by ratio and $K\in C^k$ when $|\tilde\theta_x(t_+) - x_+| > 0$.
	On the other hand, 
	\begin{multline}
		\label{lemmaeq2}
		|x_+ + \tilde\beta(\tilde\theta_x(t_+))(\tilde t_+ - t_+)|^2 
		- R^2 \\
		=  2R \tilde\beta(\tilde\theta_x(t_+))\cdot n(x_+)(\tilde t_+ 
		- t_+) 
		+ |\tilde\beta(\tilde\theta_x(t_+))|^2 (\tilde t_+ - t_+)^2.
	\end{multline}
	Substituting (\ref{lemmaeq1}) into (\ref{lemmaeq2}), we have
	\begin{equation*}
		A(\tilde t_+ - t_+)^2 + B(\tilde t_+ - t_+) + C = 0,
	\end{equation*}
	where
	\begin{align*}
		A &= |\tilde\beta(\theta_x(t_+))|^2 - K C^2\delta^2,\\
		B &= 2R \tilde\beta(\theta_x(t_+))\cdot n(x_+) 
				- 2KC\delta(R + |\tilde\theta_x(t_+) - x_+|),
				\;\mathrm{and}\\
		C &= - 2KR|\tilde\theta_x(t_+) - x_+|
				 - K|\tilde\theta_x(t_+) - x_+|^2.
	\end{align*}
	By the quadratic formula,
	\begin{equation}
	\label{t_est}
		\tilde t_+ - t_+ 
		= \frac{-B + \sqrt{B^2 - 4AC}}{2A} 
		= \frac{-2C}{B + \sqrt{B^2 - 4AC}}. 
	\end{equation}
	Notice that when $x\in\tilde\Omega$ and $x_+\in\partial\Omega_+ \cap \partial\tilde\Omega$, for $h$ sufficiently small, $\tilde\beta(\tilde\theta_x(t_+))\cdot n(x_+) > \eta$ for some fixed $\eta>0$. When $\delta$ is small, we see that $A>0$ and $B>\eta'>0$ for some $\eta'$. Thus, (\ref{thetastable}) and (\ref{t_est}) implies
	\begin{equation}
	\label{t_est2}
		\tilde t_+ - t_+ 
		\leq C' |\tilde\theta_x(t_+) - x_+|
		\leq C'' \|\beta - \tilde\beta\|_{C^0(\tilde\Omega)}.
	\end{equation}
	By taking derivatives of (\ref{t_est}),
	\begin{equation}
		\partial_x^k (\tilde t_+ - t_+) 
		= \sum_{i=0}^k \alpha_i \partial_x^i |\tilde\theta_x(t_+) - x_+|,
	\end{equation}
	where $|\alpha_i|$ is bounded when $\delta$ is small.
	
	We next need to estimate $\partial_x^i |\tilde\theta_x(t_+) - x_+|$. When $i=0$, it is directly from (\ref{thetastable}); When $i=1$, we have
	\begin{align*}
		\partial_x|\tilde\theta_x(t_+) - x_+|
		&\leq |\partial_x (\tilde\theta_x(t_+) - x_+)|\\
		&=  |(\tilde W(t_+) - W(t_+)) + (\tilde\beta(\tilde\theta_x(t_+)) 
				- \beta(\theta_x(t_+)))\partial_x t_+|\\
		&\leq C \|\beta - \tilde\beta\|_{C^1(\tilde\Omega)},
	\end{align*}
	where the inequality follows from (\ref{thetastable}) and the continuity of $\beta$. Similar estimate works for high order derivatives. Thus,
	\begin{equation}
	\label{t_est3}
		\|\tilde t_+ - t_+\|_{C^1(\tilde{\Omega})}
		\leq \|\beta - \tilde\beta\|_{C^1(\tilde{\Omega})}.
	\end{equation}
	
	Furthermore, the estimate of $|x_+ - \tilde x_+|$ follows easily.
	\begin{align*}
		|\partial_x (x_+ - \tilde x_+)|
		&= |\partial_x (\tilde\theta_x(\tilde t_+) - \theta_x(t_+))|\\
		&= |(\tilde W(\tilde t_+) - W(t_+)) 
				+ (\tilde\beta(\tilde t_+)\partial_x\tilde t_+ 
				- \beta(t_+)\partial_x t_+ )|\\
		&\leq C \|\beta - \tilde\beta\|_{C^1(\tilde{\Omega})},
	\end{align*}
	where the inequality follows from (\ref{thetastable}), (\ref{t_est3}) and the continuity of $\beta$ and $W$. High order derivatives follow similarly. This finishes the proof.
\end{proof}
  
\begin{prop} \label{stableprop}
  Let $k\geq 1$. Let $\mu$ and $\tilde{\mu}$ be solutions to
  (\ref{normvecode}) corresponding to coefficients $(\beta, 
  \gamma)$ and $(\tilde{\beta}, \tilde{\gamma})$, respectively, 
  where (\ref{betaapprox}) holds for both $\beta$ and 
  $\tilde{\beta}$.
  
  Let us define $\mu_0 = \mu|_{\partial\Omega}$ and 
  $\tilde{\mu}_0 = \tilde{\mu}|_{\partial\Omega}$, 
  thus $\mu_0, \tilde{\mu}_0\in 
  C^k(\partial\Omega)$. We also assume $h$ is sufficiently 
  small and $\Omega$ is sufficiently convex that Hypothesis 
  \ref{ConvHypo} holds for some $R<\infty$. Then there is a 
  constant $C$ such that
  \begin{multline}
      \|\mu - \tilde{\mu}\|_{C^{k-1}(\tilde{\Omega})} \\
     \leq C\|\mu_0\|_{C^{k}(\partial{\Omega}_+)}
     ( \|\beta - \tilde{\beta}\|_{C^{k}(\tilde{\Omega})}
     + \|\gamma - \tilde{\gamma}\|_{C^{k-1}(\tilde{\Omega})})
     +C\|\mu_0 - \tilde{\mu}_0\|_{C^{k}(\partial{\Omega}_+)}.
  \end{multline}
\end{prop}

\begin{proof}    
  By the method of characteristics, $\mu(x)$ is determined
  explicitly in (\ref{muchar}), while $\tilde\mu(x)$ has a 
  similar expression.
  \begin{align*}
      |\mu(x) - \tilde\mu(x)|  \leq 
      & \left|(\mu_0(x_+(x)) - \tilde\mu_0(\tilde{x}_+(x)))
             e^{-\int_0^{t_+(x)}\gamma(\theta_x(s))ds} 
             \right|\;+ \\
      & \left|\tilde\mu_0(\tilde{x}_+(x))
             \left( e^{-\int_0^{t_+(x)}\gamma(\theta_x(s))ds}
            - e^{-\int_0^{\tilde{t}_+(x)}\tilde\gamma
            (\tilde\theta_x(s))ds}
           \right)\right|
  \end{align*}
  Applying Lemma \ref{xtlemma}, we deduce that
  \begin{equation}
      |D_x^{k-1}[\mu_0(x_+(x)) - \tilde\mu_0(\tilde{x}_+(x))]|
      \leq \|\mu_0 - \tilde\mu_0\|_{C^{k-1}(\partial\Omega)}
      + C\|\mu_0\|_{C^{k-1}(\partial\Omega_+)}
        \|\beta - \tilde{\beta}\|_{C^{k-1}(\tilde{\Omega})}.
  \end{equation}
  This proves the $\mu_0(x_+(x))$ is stable. To consider the 
  second term, by the Leibniz 
  rule it is sufficient to prove the stability result for
  $\int_0^{t_+(x)}\gamma(\theta_x(s))ds$.
  
  Assume without loss of generality that $t_+(x)<\tilde{t}_+(x)$. 
  Then we have, applying (\ref{thetastable}),
  \begin{align*}
      \int_0^{t_+(x)} [ \gamma(\theta_x(s)) - 
      \tilde\gamma(\tilde\theta_x(s))]ds
      &= 
      \int_0^{t_+(x)} [
      ( \gamma(\theta_x(s)) - \gamma(\tilde\theta_x(s))) +
      ( \gamma - \tilde\gamma)(\tilde\theta_x(s))] ds\\
      &\leq
      C\|\gamma\|_{C^{0}(\tilde{\Omega})}
      \|\beta - \tilde{\beta}\|_{C^{0}(\tilde{\Omega})}
      + C\|\gamma - \tilde\gamma\|_{C^{0}(\tilde{\Omega})}.
  \end{align*}
  Derivatives of order $k-1$ of the above expression
  are uniformly bounded since 
  $t_+(x)\in C^{k-1}(\tilde\Omega)$, $\gamma$ has $C^k$
  derivatives bounded on $\tilde\Omega$ and $\theta_x(t)$ is
  stable as in (\ref{thetastable}).
  
  It remains to handle the term
  $$
      v(x) = \int_{t_+(x)}^{\tilde{t}_+(x)}
             \tilde\gamma(\tilde\theta_x(s))ds.
  $$
	$\tilde\beta$ and $\tilde\gamma$ are of class $C^k(\Omega)$, then so is the function $x\rightarrow\tilde\gamma(\tilde\theta(s))$. Derivatives of order $k-1$ of $v(x)$ involve terms of size $\tilde t_+(x) - t_+(x)$ and terms of form
	\begin{equation*}
		D_x^m\big(
		\tilde t_+ D_x^{k-1-m} \tilde\gamma(\tilde\theta_x(\tilde t_+)) 
		- t_+ D_x^{k-1-m} \tilde\gamma(\tilde\theta_x(t_+))
		\big),
		\quad 0\leq m\leq k-1.
	\end{equation*}
Since the function has $k-1$ derivatives that are Lipschitz continuous, we thus have
	\begin{equation*}
		|D_x^{k-1} v(x)| 
		\leq C \|\tilde t_+ - t_+\|_{C^{k-1}(\Omega)}.
	\end{equation*}
The rest of the proof follows Lemma \ref{xtlemma}.
\end{proof}
  
  With all prepared, we are ready to  prove the following stability result.
  
\begin{thm} \label{stabthm}
	Let $k\geq3$. Assume that $(\mu, q)$ and $(\tilde \mu, \tilde q)$ are in $\mathcal{P}$ and that $|g_j-\check u_j|_{\partial\Omega}$, $j=1,2$, are sufficiently small. Let $\mu,\tilde\mu$ be solutions of (\ref{normvecode}) with $(\beta,\gamma), (\tilde\beta,\tilde\gamma)$ and  $h$ sufficiently small. Assume on the boundary,$\mu_0 = \tilde\mu_0$.  Then we have
that
\begin{equation}
	\|\mu - \tilde\mu\|_{C^{k-1}(\tilde\Omega)}
	+ \|q - \tilde q\|_{C^{k-3}(\tilde\Omega)}
	\leq C \|d - \tilde d\|_{(C^{k}(\tilde\Omega))^2}.
\end{equation}
\end{thm}

\begin{proof}
	The first part follows directly from (\ref{betagamma}) and Proposition \ref{stableprop}. This provides a stability result for $\nu = 1/\mu$ and thus for $u_j = \nu d_j$. Notice $\check u_j$ is non-vanishing on $\tilde\Omega$. So when choosing $|g_j-\check u_j|_{\partial\Omega}$ sufficiently small, the arguments in (\ref{gconstrain}) and (\ref{uconstrain}) show that $u_j$ is non-vanishing in $\tilde\Omega$. Thus $\Delta u_j + u_j q = 0$ gives the stability control of $q$.
\end{proof}
  
\subsection{Stability Result for $2n$ complex observations}
  We now consider the case that $4n$ real-valued observations 
  are taken, with which we can construct $2n$ sets of 
  complex-valued boundary and internal data, denoted 
  as $g_{1,2}^j=\{g_1^j, g_2^j\}$ and $d_{1,2}^j=\{d_1^j, d_2^j\}$.
  For the rest of this section, we choose  $j=1,\ldots,n$.
  
  Let $H$ be a hyperplane in 
  $\RR^n\backslash\overline{\ch(\Omega)}$.
  Choose $x^j\in H$, such that $\{x^j-x^1\}$ form a basis
  of $H$, thus $\mathrm{span}\{x^j-x^1\}$ has dimension 
  $n-1$. Then for $\forall x\in\Omega$, $\{x^j-x\}$
  form a basis of $\RR^n$. In fact, since $x\notin H$, 
  $$
      \mathrm{span}\{x^j-x\} =
      \mathrm{span}\{x^j-x^1, x^1-x\}
  $$
  has dimension $m$.

  Define 
  $$
      \varphi^j= \log|x-x^j|,\qquad
      \psi = \mathrm{dist}\left(\frac{x-x^j}{|x-x^j|},\omega\right),
      \qquad j=1,\ldots,n.
  $$
  Then the matrix
  $
      {B}_\varphi :=(\nabla\varphi^j)
  $
  is invertible.

  Corresponding to $x^j$, $\varphi^j$ and $\psi^j$,
  we can define the front and back sides of the boundary,
  $\partial\Omega_+^j$ and $\partial\Omega_-^j$, by
  (\ref{frontbackbd}), and define the CGO solutions,
   $\check{u}^j_{1,2}$,  
  by (\ref{CGOu12}).
  
  Let us now choose boundary conditions $g^j_{1,2}$ close to
  $\check{u}_{1,2}^j$ on $\partial\Omega$, precisely,
  by (\ref{gconstrain}). With internal data $d_{1,2}^j$, we can 
  define $\beta^j$ by (\ref{betagamma}) and (\ref{normbeta}).
  
  Proposition \ref{uniquemu} shows that the matrix $B = (\beta^j)$
  is close to an invertible matrix $B_\varphi$ on a subregion
  $\tilde{\Omega}$. Therefore, we can make $B$ invertible on
  $\tilde{\Omega}$ 
  with inverse of class 
  $C^k(\tilde{\Omega})$ by choosing $h$  
  sufficiently small. Consequently, (\ref{normvecode}) can be 
  rewrite
  \begin{equation}
  	\label{4nstab1}
      \nabla\mu + \Lambda\mu = 0,
  \end{equation}
  where $\Lambda$ is a vector-valued function in 
  $(C^k(\tilde\Omega))^n$.
  Finally, the construction of $\Lambda$ is stable under small
  perturbations in the data $d^j$. Indeed, let $\Gamma$ and 
  $\tilde\Gamma$ be two vector fields constructed from the
  internal measurements $d_{1,2}^j$ and $\tilde d_{1,2}^j$, 
  respectively. Then when $h$ is sufficiently small,
  \begin{equation}
  	\label{4nGammastab}
  	\|\Gamma - \tilde\Gamma\|_{(C^k(\tilde\Omega))^n}
  	\leq \|d - \tilde d\|_{(C^{k+1}(\tilde\Omega))^n}.
  \end{equation}

   Now we consider equation (\ref{4nstab1}) with boundary condition
   $\mu = \mu_0$. Assume $\Omega$ is bounded and connected and 
   $\partial\Omega$ is smooth. Let $x\in\tilde\Omega$. Find a 
   smooth curve from $x$ to a point on the boundary.
   Restricted to this curve, (\ref{4nstab1})
   is a stable ordinary differential equation. Keep the curve fixed.
   Let $\mu, \tilde\mu$ be solutions to (\ref{4nstab1}) with
   respect to $\Gamma, \tilde\Gamma$, respectively. Assume
   $\mu_0 = \tilde\mu_0$ on $\partial\Omega$. By solving the equation
   explicitly and applying (\ref{4nGammastab}), we find that
   \begin{equation}
   		\label{4nmustab}
   		\|\mu - \tilde\mu\|_{C^{k}(\tilde\Omega)}
   		\leq \|d - \tilde d\|_{C^{k+1}(\tilde\Omega)}.
   \end{equation}
   
   We can now state the main theorem of this section.

\begin{thm}
	Let $k\geq 2$. Assume that $(\mu, q)$ and $(\tilde \mu, \tilde q)$ are in $\mathcal{P}$ and that we have $2n$ well-chosen complex valued measurement, such that	$|g_{1,2}^j-\check u_{1,2}^j|_{\partial\Omega}$, $j=1,\ldots,n$, are sufficiently small. Let $\mu,\tilde\mu$ be solutions of (\ref{4nstab1}) with $\Gamma, \tilde\Gamma$ and  $h$ sufficiently small. Assume on the boundary, $\mu_0 = \tilde\mu_0$.  Then we have
that
	\begin{equation}
		\|\mu - \tilde\mu\|_{C^{k}(\tilde\Omega)}
		+ \|q - \tilde q\|_{C^{k-2}(\tilde\Omega)}
		\leq C \|d - \tilde d\|_{(C^{k+1}(\tilde\Omega))^2}.
	\end{equation}
\end{thm}

\begin{proof}
	The first result is directly from (\ref{4nmustab}). The proof of the stability of $q$ is exactly that same as in the proof of Theorem \ref{stabthm}.
\end{proof}

\section{Proof of Main results}
	In this section, we return to the inverse diffusion problem with internal data (\ref{diff_eq}) and prove the main theorems. We will invert the standard Liouville change of variables mentioned at the beginning of section \ref{ISID}. Recall that by defining $v = \sqrt{D}u$, we had (\ref{Liou1}, \ref{Liou2}, \ref{Liou3}), which induce
	\begin{equation}
		\label{Liou4}
		-\Delta\sqrt{D} - q\sqrt{D} = \mu.
	\end{equation}
	Section \ref{ISID} allows us to reconstruct $\mu$ and $q$, while $D$ is given on $\partial\Omega$. Thus we can solve for $\sqrt{D}$ from (\ref{Liou4}) and then $\sigma_a = \mu\sqrt{D}$.
	
	The uniqueness of thus a solution $\sqrt{D}$ of (\ref{Liou4}) is based on that $0$ is not an eigenvalue of $\Delta + q$. It is enough to prove that $(D,\sigma_a)\in\mathcal{M}$ implies $(q,\mu)\in\mathcal{P}$. Indeed, the inverse of $\Delta + q$ is compact and $\sqrt{D}\in Y = H^{\frac{n}{2}+k+2+\epsilon}(\Omega) \subset C^{k+2}(\bar{\Omega})$. These imply $\Delta + q$ is surjective and thus does not have 0 as an eigenvalue by the Fredholm alternative.
	
	So far we proved the unique reconstruction of $D,\sigma_a$ form internal data for well-chosen partial boundary illuminations as stated in Theorem \ref{main1}.
	
\begin{proof}[Proof of Theorem \ref{main2}]
	Since $\sigma = \mu\sqrt{D}$, we only need to prove the stability of $\sqrt{D}$. When $k\geq3$, we have the stability of the reconstructions of $q\in C^k$ and $\mu\in C^{k+1}$ in 
Theorem \ref{stabthm},
while the boundary illuminations are chosen from a open set in $C^{k,\alpha}$ with $\alpha > 1/2$. According to (\ref{Liou4}), we have
\begin{equation}
	-(\Delta + q)(\sqrt{D} - \sqrt{\tilde D}) = \mu - \tilde\mu
	+ (q - \tilde q)\sqrt{\tilde D}.
\end{equation}
By elliptic regularity, we deduce that $(\sqrt{D} - \sqrt{\tilde D})$ is bounded in $C^k(\bar\Omega)$, and hence the theorem.
\end{proof}

The proof of Theorem \ref{main3} will be the same as above.

\end{document}